\documentclass{amsart}

\usepackage{amssymb,amscd,amsmath,hyperref,color,enumerate,tikz-cd}
\usepackage[all,cmtip]{xy}

\newcommand{\N}{{\ensuremath{\mathbb{N}}}}

\newcommand{\C}{{\ensuremath{\mathbb{C}}}}
\def\BB{\mathbb{B}}

\def\KK{\mathbb{K}}
\def\MM{\mathrm{M}}

\def\la{\langle}
\def\ra{\rangle}

\def\C{\mathbb{C}}
\def\N{\mathbb{N}}

\newcommand\CB{\mathop{\mathrm{CB}}\nolimits}

\def\L{\mathcal{L}}
\def\M{\mathcal{L}_{\mathrm{ext}}}
\def\MM{\mathrm{M}}

\def\wt{\widetilde}
\newtheorem{proposition}{Proposition}[section]
\newtheorem{lemma}[proposition]{Lemma}

\newtheorem{theorem}[proposition]{Theorem}

\theoremstyle{definition}
\newtheorem{remark}[proposition]{Remark}

\newtheorem{definition}[proposition]{Definition}
\newtheorem{example}[proposition]{Example}

\numberwithin{equation}{section}

\begin{document}

\title[]{Elementary operators on Hilbert modules over prime $C^*$-algebras}

\author{Ljiljana Aramba\v{s}i\'c}

\address{Department of Mathematics, Faculty of Science, University of Zagreb, Bijeni\v{c}ka 30,
10000 Zagreb, Croatia}

\email{arambas@math.hr}

\author{Ilja Gogi\'c}

\address{Department of Mathematics, Faculty of Science, University of Zagreb, Bijeni\v{c}ka 30,
10000 Zagreb, Croatia}

\email{ilja@math.hr}

\date{\today}

\thanks{This work has been fully supported by the Croatian Science Foundation under the project IP-2016-06-1046.}



\begin{abstract}
Let $X$ be a right Hilbert module over a $C^*$-algebra $A$ equipped with the canonical operator space structure. We define an elementary operator on $X$ as a map $\phi : X \to X$ for which there exists a finite number of elements $u_i$ in the $C^*$-algebra $\mathbb{B}(X)$ of adjointable operators on $X$ and $v_i$ in the multiplier algebra $M(A)$ of $A$ such that $\phi(x)=\sum_i u_i xv_i$ for $x \in X$. If $X=A$ this notion agrees with the standard notion of an elementary operator on $A$. In this paper we extend Mathieu's theorem for elementary operators on prime $C^*$-algebras by showing that the completely bounded norm of each elementary operator on a non-zero Hilbert $A$-module $X$ agrees with the Haagerup norm of its corresponding tensor in  $\mathbb{B}(X)\otimes M(A)$ if and only if $A$ is a prime $C^*$-algebra.
\end{abstract}

\maketitle


\section{Introduction}

An operator on a  $C^*$-algebra $A$ is called an elementary operator if it can be expressed as a finite sum of two-sided multiplications $M_{a,b} : x \mapsto
axb$, where $a$ and $b$ are elements of the multiplier algebra $M(A)$. In other words, an elementary operator on $A$ is a map $\phi : A \to A$ of the form
$\phi : x\mapsto \sum_{i}a_ixb_i$ for some finite collections of  $a_i,b_i\in M(A).$ Obviously, such a representation of an elementary operator is not unique.

It is well-known that elementary operators on $C^*$-algebras are completely bounded mappings with the following estimate for their cb-norm:
$$
\left\|\sum_{i} M_{a_i,b_i}\right\|_{cb}\leq \left\|\sum_i a_i \otimes b_i\right\|_h,
$$
where $\|\cdot\|_h$ is the Haagerup tensor norm on the algebraic
tensor product $M(A) \otimes M(A)$, i.e.
\[
\|t\|_h = \inf \left\{\left\|\sum_{i} u_iu_i^*\right\|^{\frac{1}{2}}\left\|\sum_{i} v_i^*v_i\right\|^{\frac{1}{2}} \ : \ t=\sum_{i} u_i \otimes v_i\right\}.
\]
Hence, if $\CB(A)$ denotes the set of all completely bounded maps on $A$, the above inequality ensures that the mapping
\[
(M(A) \otimes M(A), \|\cdot\|_h) \to (\CB(A), \|\cdot\|_{cb}) \quad \mbox{given by} \quad \sum_{i} a_i \otimes b_i \mapsto \sum_{i} M_{a_i,b_i}
\]
is a well-defined contraction. Its continuous extension to the Haagerup tensor product $M(A) \otimes_h M(A)$ (which is the completion of $M(A)\otimes M(A)$ in $\|\cdot\|_h$) is known as the canonical contraction from $M(A) \otimes_h M(A)$ to $\CB(A)$ and is denoted by $\Theta_A$.

An interesting and a non-trivial question is to characterize the case when $\Theta_A$ is isometric or injective. The obvious necessary condition for the injectivity of $\Theta_A$ is that $A$ is a prime $C^*$-algebra. It turns out that the primeness of $A$ is also a sufficient condition for $\Theta_A$  to be isometric. First, Haagerup  showed in \cite{Haa} that $\Theta_A$ is isometric if $A$ is the $C^*$-algebra of all bounded linear operators on a Hilbert space. Then Chatterjee and Sinclair  showed in \cite{CS} that $\Theta_A$ is isometric if $A$ is a separably-acting von Neumann factor. Finally, Mathieu completed the answer to this problem \cite[Proposition 5.4.11]{AM}:

\begin{theorem}[Mathieu]\label{thetaiso}
Let $A$ be a $C^*$-algebra. The following conditions are equivalent:
\begin{itemize}
\item[(i)] $\Theta_A$ is isometric.
\item[(ii)] $\Theta_A$ is injective.
\item[(iii)] $A$ is a prime $C^*$-algebra.
\end{itemize}
\end{theorem}
If a $C^*$-algebra $A$ is unital, but not necessarily prime, one can construct a central Haagerup tensor product $A \otimes_{Z,h} A$ and consider the induced contraction $\Theta_{A}^Z : A \otimes_{Z,h} A \to \CB(A)$. The analogous questions about $\Theta_{A}^Z$ were treated in \cite{Som,AST1,AST2}.

It is also an interesting problem to consider which classes of maps (like derivations or automorphisms) on $C^*$-algebras can be approximated by two-sided multiplications or elementary operators in the operator or completely bounded norm. For results on this subject we refer to \cite{Gog1, GT} and the references within.

\smallskip

The purpose of this paper is to extend Theorem~\ref{thetaiso} to the class of operators on Hilbert $C^*$-modules which generalize elementary operators on $C^*$-algebras.

\section{Preliminaries}
Throughout the paper $A$ will be a $C^*$-algebra. By an ideal of $A$ we always mean a closed two-sided ideal. An ideal $I$ of $A$ is said to be \emph{essential} if for any $a \in A$, $aI=\{0\}$ (or $Ia=\{0\}$) implies $a=0$.

A $C^*$-algebra $A$ is said to be \emph{prime} if the product of any two non-zero ideals of $A$ is non-zero. Equivalently, $A$ is prime for $a,b\in A$ such that $aAb=\{0\}$ it follows that $a=0$ or $b=0$ (see e.g. \cite[Lemma~2.17]{INCA}).

A \emph{Hilbert $C^*$-module over} $A$ (or a \emph{Hilbert $A$-module}) is a right $A$-module $X$ equipped with an $A$-valued inner product $\la \cdot,\cdot \ra : X\times X \to A$ such that $X$ is a Banach space with respect to the norm defined by $\|x\|=\|\la x,x\ra\|^{\frac{1}{2}}.$ Recall that the inner product on $X$ has the properties
\begin{enumerate}
\item $\la x,\alpha y+\beta z\ra=\alpha \la x,y\ra +\beta \la x,z\ra$,
\item $\la x,ya\ra=\la x,y\ra a$,
\item $\la x,y\ra=\la y,x\ra^*$,
\item $\la x,x\ra \geq 0$; $\la x,x\ra =0 \Leftrightarrow x=0$,
\end{enumerate}
that are satisfied for all $x,y,z\in X$, $a\in A$ and $\alpha,\beta\in \Bbb C.$ In a similar way a left Hilbert $A$-module is defined; the only differences are that we have a left module action and an inner product is linear and $A$-linear in the first variable instead of in the second variable.

For a  Hilbert $A$-module $X$ we denote by $\la X, X \ra$ the closed linear span of the set $\{\la x,y\ra : x,y\in X\}$. Clearly,  $\la X, X \ra$ is an ideal of $A$. If $\la X, X \ra=A$, $X$ is said to be \emph{full}. We will say that $X$ is \emph{essentially full} if $\la X, X \ra$ is an essential ideal of $A$.

Every $C^*$-algebra can be  regarded as a Hilbert $C^*$-module over itself with respect to the inner product $\la a,b\ra=a^*b.$ Also, if $I$ is an ideal in a $C^*$-algebra $A$ then $I$ can be regarded as a Hilbert $A$-module with the same inner product. Further, if $X_1,\ldots,X_n$ are Hilbert $A$-modules, then $X_1\oplus\ldots\oplus X_n$ is a Hilbert $A$-module with respect to the module action given as
$$
(x_1\oplus\ldots\oplus x_n)a=x_1a\oplus\ldots\oplus x_na
$$
and the inner product
$$\la x_1\oplus\ldots\oplus x_n,y_1\oplus\ldots\oplus y_n\ra=\sum_{i=1}^n\la x_i,y_i\ra.
$$

If $X$ and $Y$ are Hilbert $A$-modules we denote by $\Bbb B(X,Y)$ the Banach space of all \emph{adjointable operators} from $X$ to $Y,$ that is, those $u:X\to Y$ for which there is
$u^*:Y\to X$ with the property
$$\la ux,y\ra=\la x,u^*y\ra \qquad  \forall x\in X, \, y\in Y.$$
It is well-known that all adjointable operators are bounded and  $A$-linear (i.e. $u(xa)=(ux)a$ for all $x \in X$ and $a \in A$). By $\Bbb K(X,Y)$ we denote the Banach subspace of $\Bbb B(X,Y)$ generated by the maps
$$\theta_{y,x}:X\to Y,\qquad z \mapsto y\la x,z\ra, \qquad$$
where $x\in X$ and $y\in Y$ are arbitrary.
If $X=Y$ we write $\Bbb B(X)$ and $\Bbb K(X)$ (or $\BB_A(X)$ and $\KK_A(X)$ when we want to emphasize the underlying $C^*$-algebra $A$), and these are $C^*$-algebras. Moreover, $\Bbb B(X)$ is the multiplier $C^*$-algebra of $\Bbb K(X)$ (\cite[Corollary~2.54]{RW}). If we regard a $C^*$-algebra $A$ as a Hilbert module over itself, then $\Bbb B(A)$ is actually the multiplier $C^*$-algebra $M(A)$ of $A$.


If $X$ is a Hilbert $A$-module then, regarding $A$ as a Hilbert $A$-module, $A\oplus X$ becomes a Hilbert $A$-module in above-mentioned way, so the $C^*$-algebras $\Bbb K(A\oplus X)$ and $\Bbb B(A\oplus X)$ are well defined.
The first of them, i.e. $\Bbb K(A\oplus X)$, is known as the \emph{linking algebra} of $X$; we denote it by $\L(X)$ (\cite[p.~350]{BGR}).
Then we can write
$$\L(X)=\begin{bmatrix}
\Bbb K(A) & \Bbb K(X,A) \\
\Bbb K(A,X) & \Bbb K(X) \\
\end{bmatrix}
=\left\{\begin{bmatrix}
T_a & l_y \\
r_x & u \\
\end{bmatrix}: \, a\in A,\,  x,y\in X, \, u\in \Bbb K(X)
\right\},$$
where $T_a(b)=ab$ and $r_x(b)=xb$ for all $b\in A$, while $l_y(z)=\la y,z\ra$ for all $z\in X$.  Thereby, $a\mapsto T_a$ is an isomorphism of $C^*$-algebras $A$ and $\Bbb K(A),$ $y\mapsto l_y$ is an isometric conjugate linear isomorphism between Banach spaces $X$ and $\Bbb K(X,A),$
and $x\mapsto r_x$ is an isometric linear isomorphism between Banach spaces $X$ and $\Bbb K(A,X).$

For more details about Hilbert $C^*$-modules we refer the reader to \cite{La, MT, RW, WO}.

If $X$ is an operator space we write $\CB(X)$ for the Banach algebra of all completely bounded maps on $X$. For details about operator spaces, their tensor products and completely bounded maps we refer to \cite{BLM, ER, SS}.


\section{Results}

Let $X$ be a Hilbert $A$-module. Besides the linking algebra $\L(X),$ we need another subalgebra of $\BB (A\oplus X),$ larger than $\L(X).$

We define an \emph{extended linking algebra} of $X$ as
$$\M(X)=\begin{bmatrix}
\BB(A) & \Bbb K(X,A) \\
\Bbb K(A,X) & \BB(X) \\
\end{bmatrix}=
\left\{\begin{bmatrix}
T_v & l_y \\
r_x & u \\
\end{bmatrix}:v\in M(A), \, x,y\in X, \, u\in \Bbb B(X)
\right\},$$
where, similarly as before, for $v \in M(A)$, $T_v : A \to A$ is defined by $T_v(a)=va$.

Let us first show that $\M(X)$ is a $C^*$-algebra. For that we shall need the following remark.

\begin{remark}\label{rem:multmod}
Let $X$ be a Hilbert $A$-module. If $B$ is any $C^*$-algebra that contains $A$ as an ideal, then $X$ can be also regarded as a Hilbert $B$-module with respect to the same inner product (which takes values in $A\subseteq B$), while the right action of $B$ on $X$ is defined as follows. For $x \in X$, $a \in A$ and $b \in B$, set
$$(xa)b:=x(ab)$$
(see e.g. \cite[8.1.4 (4)]{BLM}). Obviously, $\BB_B(X)=\BB_A(X)$ and $\KK_A(X)=\KK_B(X)$, so all $u \in \BB_A(X)$ are also $B$-linear. In particular, by taking $B=M(A)$, any Hilbert $A$-module $X$ can be regarded as a Hilbert $M(A)$-module.
\end{remark}

\vspace{1mm}

\begin{lemma}\label{lem:ela} Let $X$ be a Hilbert $A$-module.
$\M(X)$ is a $C^*$-subalgebra of $\BB(A \oplus X)$ which contains $\L(X)$ as an essential ideal.
\end{lemma}
\begin{proof}
Clearly $\M(X)$ is a linear subspace of $\BB(A \oplus X)$.
If
$$S=\begin{bmatrix}
T_v & l_y \\
r_x & u \\
\end{bmatrix}\in \M(X),$$
one can easily verify that the adjoint of $S$ in $\BB(A \oplus X)$ is given by
$$\begin{bmatrix}
T_{v^*} & l_x \\
r_y & u^* \\
\end{bmatrix},$$
so $S^* \in \M(X)$.  Further, for all $v_1,v_2 \in M(A)$, $x_1,x_2,y_1,y_2 \in X$ and $u_1,u_2 \in \BB(X)$ we have
$$\begin{bmatrix}
T_{v_1} & l_{y_1} \\
r_{x_1} & u_1 \\
\end{bmatrix}
\begin{bmatrix}
T_{v_2} & l_{y_2} \\
r_{x_2} & u_2 \\
\end{bmatrix} =
\begin{bmatrix}
T_{v_1v_2+\la y_1,x_2\ra} & l_{y_2v_1^*+u_2^*y_1} \\
r_{x_1v_2+u_1x_2} & \theta_{x_1,y_2}+u_1u_2
\end{bmatrix}\in \M(X),$$
since $X$ can be regarded as a Hilbert $M(A)$-module (Remark \ref{rem:multmod}) and hence $y_2v_1^*,x_1v_2\in X$. This shows that $\M(X)$ is a self-adjoint subalgebra of $\BB(A \oplus X)$. Using the similar arguments as in the proof of \cite[Lemma~3.20]{RW} we also conclude that $\M(X)$ is norm closed and hence a $C^*$-subalgebra of $\BB(A \oplus X)$.

Finally, using the fact that $\L(X)$ is an essential ideal of (its multiplier $C^*$-algebra) $\BB(A\oplus X)$, we conclude that $\L(X)$ is an essential ideal of $\M(X)$.
\end{proof}

In the introduction we gave the notion of essentially full Hilbert modules: a Hilbert $A$-module $X$ is essentially full if $\la X, X \ra$ is an essential ideal of $A$.
As we show in the next lemma, essential fullness guarantees some kind of nondegeneracy of $X$ regarded as a Hilbert $C^*$-module over any $C^*$-algebra which contains $\la X,X\ra$ as an essential ideal.

\begin{lemma}\label{lem:essfull}
For a non-zero Hilbert $A$-module $X$ the following conditions are equivalent:
\begin{itemize}
\item[(i)] $X$ is essentially full.
\item[(ii)] For each non-zero element $a \in A$ there exists $x \in X$ such that
$xa \neq 0$.
\end{itemize}
\end{lemma}
\begin{proof}
(i) $\Longrightarrow$ (ii). Assume $X$ is essentially full and  let $a\in A$ be such that
$xa=0$ for all $x \in X$. Then
$$\la y,x \ra a= \la y,xa \ra =0 \qquad \forall x,y \in X,$$
which implies $\la X, X\ra a=\{0\}$. Since $X$ is essentially full, we conclude that $a=0$.

\smallskip

(ii) $\Longrightarrow$ (i). Let $a \in A$, $a \neq 0$. By assumption, there exists $x \in X$ such that $xa \neq 0$. Then 
$$\la xa, x \ra a = \la xa, xa \ra \neq 0,$$
so  $\la X , X \ra a \neq \{0\}$. Therefore, $X$ is essentially full. 
\end{proof}

In the following proposition we give several equivalent descriptions of Hilbert $C^*$-modules over prime $C^*$-algebras.

\begin{proposition}\label{prop:prime} Let $X$ be a non-zero Hilbert $A$-module. The following conditions are equivalent:
\begin{itemize}
\item[(i)] $A$ is prime.
\item[(ii)] $X$ is essentially full and $\KK(X)$ is prime.
\item[(iii)] The linking algebra $\L(X)$ is prime.
\item[(iv)] The extended linking algebra $\M(X)$ is prime.
\item[(v)] If $a\in A$ and $u\in \KK(X)$ are such that $uxa=0$ for all $x\in X$ then $a=0$ or $u=0$.
\item[(vi)]  $X$ is essentially full and if $x_1,x_2\in X$ are such that $x_1\la x,x_2\ra=0$ for all $x\in X$ then $x_1=0$ or $x_2=0$.
\end{itemize}
\end{proposition}

\begin{proof}
(i) $\Longrightarrow$ (ii), (iii). Assume that $A$ is prime. Then any non-zero (two-sided) ideal of $A$ is essential (see e.g. \cite[Lemma~1.1.2]{AM}), so in particular $X$ is essentially full. Observe that, in order to get that $\mathcal{L}(X)$ is prime, it is enough to show that if $A$ is prime then $\Bbb K(X)$ is also prime. Namely, the linking algebra $\mathcal{L}(X)$ is defined as $\Bbb K(A\oplus X).$ Since $A\oplus X$ is a Hilbert $C^*$-module over the same $C^*$-algebra $A$, it will then follow that $\mathcal{L}(X)$ is prime whenever $A$ is prime.

Assume there exist non-zero $u_1,u_2\in \KK(X)$ such that $u_1\KK(X)u_2=\{0\}$. Then there are $x_1,x_2\in X$ such that $u_1x_1 \neq 0$ and $u_2x_2\neq 0.$ By assumption,
$$u_1\,\theta_{x_1a,u_2x_2}\,u_2=0 \qquad \forall a \in A.$$
Then
\begin{eqnarray*}
\la u_1x_1,u_1x_1\ra a\la u_2x_2,u_2x_2\ra&=&\la u_1x_1,u_1(x_1 a\la u_2x_2,u_2x_2\ra)\ra\\
&=&\la u_1x_1,(u_1\, \theta_{x_1a,u_2x_2}\,u_2)(x_2) \ra\\
&=&0
\end{eqnarray*}
for all $a\in A,$ which is a contradiction with the assumption that $A$ is prime, since both $\la u_1x_1,u_1x_1\ra$ and $\la u_2x_2,u_2x_2\ra $ are non-zero. Therefore, $u_1\KK(X)u_2=\{0\}$ can happen only when  $u_1=0$ or $u_2=0$, which shows that $\Bbb K(X)$ is prime.

\smallskip

(ii)  $\Longrightarrow$ (i). Assume that $X$ is essentially full and that $A$ is not prime. Then there exist non-zero elements $a_1,a_2\in A$ such that $a_1Aa_2=\{0\}$. Then by Lemma \ref{lem:essfull} there are $x_1, x_2\in X$ such that $x_1a_1\ne 0$ and $x_2 a_2 \ne 0.$ By assumption,
$$a_1\la x_1a_1,ux_2\ra a_2=0 \qquad \forall u \in \KK(X).$$
Then for all $x \in X$ and $u \in \KK(X)$ we have
\begin{eqnarray*}
(\theta_{x_1a_1,x_1a_1} \, u \, \theta_{x_2a_2,x_2a_2})(x)&=&x_1a_1\la x_1a_1,u\,\theta_{x_2a_2,x_2a_2}(x)\ra\\
&=& x_1a_1\la x_1a_1,u(x_2a_2\la x_2a_2,x\ra)\ra\\
&=&x_1a_1\la x_1a_1,u x_2\ra a_2\la x_2a_2,x\ra\\
&=& x_1(a_1\la x_1a_1,ux_2\ra a_2)\la a_2x_2,x\ra\\
&=& 0.
\end{eqnarray*}
Thus,
$$\theta_{x_1a_1,x_1a_1}\,\KK(X)\,\theta_{x_2a_2,x_2a_2}=\{0\}.$$
Since both $\theta_{x_1a_1,x_1a_1}$ and $\theta_{x_2a_2,x_2a_2}$ are non-zero, we conclude that $\KK(X)$ is not prime.

\smallskip

(iii) $\Longrightarrow$ (iv). This follows directly from Lemma \ref{lem:ela} and the fact that any $C^*$-algebra that contains a prime essential ideal must be prime itself.

\smallskip

(iv) $\Longrightarrow$ (v). Assume that $\M(X)$ is prime. Then for non-zero elements $a_0\in A$ and $u_0 \in \KK(X)$ there are elements $v \in M(A)$, $x,y \in X$ and $u \in \BB(X)$ such that
$$0 \neq
\begin{bmatrix}
0 & 0 \\
0 & u_0 \\
\end{bmatrix}\begin{bmatrix}
T_v & l_y \\
r_x & u \\
\end{bmatrix}
\begin{bmatrix}
T_{a_0} & 0\\
0 & 0 \\
\end{bmatrix}
=\begin{bmatrix}
0 & 0 \\
r_{u_0xa_0} & 0 \\
\end{bmatrix}.$$
Thus, $u_0xa_0\neq 0$ for some $x\in X$.

\smallskip

(v)  $\Longrightarrow$ (vi). Suppose first that there exists $a\in A$, $a\neq 0$, such that $xa=0$ for all $x\in X.$ Then $uxa=0$ for all $x\in X$ and $u\in \KK (X).$ By assumption, it follows that $u=0$ for all $u\in \KK(X),$ which is not since $X\ne \{0\}.$ Therefore, $X$ is essentially full.

Let $x_1,x_2 \in X$ be such that $x_1 \la x, x_2 \ra=0$ for all $x\in X$. Then
$$\theta_{x_1,x_1}(x) \la x_2, x_2\ra = x_1 \la x_1, x\ra \la x_2 , x_2 \ra= x_1 \la x_2 \la x,x_1 \ra ,x_2\ra=0 \qquad \forall x\in X.$$
Hence, by assumption, $\theta_{x_1,x_1}=0$ or $\la x_2, x_2\ra=0$, that is, $x_1=0$ or $x_2=0$.

\smallskip

(vi)  $\Longrightarrow$ (i). Assume (vi) holds but $A$ is not prime. Then there are non-zero elements $a_1,a_2\in A$ such that $a_1Aa_2=\{0\}$. By assumption $X$ is essentially full, so by Lemma \ref{lem:essfull} there exist $x_1, x_2\in X$ such that $x_1a_1\ne 0$ and $x_2 a_2 \ne 0$. But then
$$x_1a_1 \la x,x_2a_2 \ra=x_1a_1 \la x,x_2\ra a_2=0 \qquad \forall x \in X,$$
which contradicts our assumption.
\end{proof}

\begin{remark}
In particular, Proposition \ref{prop:prime} shows (probably the well-known fact) that the primeness is an invariant property under  Morita equivalence (see e.g. \cite[Chapter~3]{RW}). Indeed, if $X$ is an $A-B$ imprimitivity bimodule, then by definition $X$ is full both as a left Hilbert $A$-module and as a right Hilbert $B$-module. Then $A\cong \KK(X)$ by \cite[Proposition~3.8]{RW}, so the equivalence of (i) and (ii) in Proposition \ref{prop:prime} says that $A$ is prime if and only if $B$ is prime.  For the other interesting properties that are invariant under Morita equivalence we refer to \cite{HRW}.
\end{remark}

The next simple example demonstrates the necessity of the assumption that $X$ if essentially full in both conditions (ii) and (vi) of Proposition \ref{prop:prime}.

\begin{example}
Let $A$ be any non-prime $C^*$-algebra that contains a prime non-zero ideal $I$ (e.g. $A=\C\oplus \C$ and $I=\C \oplus \{0\}$).  Consider $X=I$ as a Hilbert $A$-module in the usual way. Then $\KK(X)=I$ is a prime $C^*$-algebra, while $A$ is not.

Further, if $x_1,x_2\in X$ satisfy $0=x_1\la x,x_2\ra=x_1x^*x_2$ for all $x\in X$, the primeness of $I$ implies $x_1=0$ or $x_2=0$. Therefore, the second condition in (vi) is satisfied, but (i) does not hold.
\end{example}

\medskip

If $X$ is a Hilbert $A$-module, we can introduce the operator space structure on $X$ via the operator space structure of its linking algebra $\L(X)$ (or extended linking algebra $\M(X)$), after identifying $X$ as the $2-1$ corner in $\L(X)$ (or $\M(X)$), via the isometric isomorphism $X \cong \KK(A,X)$, $x \mapsto r_x$. That is, for all $n \in \N$ and $\begin{bmatrix} x_{ij}  \end{bmatrix}\in \MM_n(X)$ we define
$$\left\|\begin{bmatrix}
x_{ij}
\end{bmatrix}
\right\|_{\MM_n(X)}:=\left\|\begin{bmatrix}\begin{bmatrix} 0 & 0 \\
r_{x_{ij}} & 0\end{bmatrix}\end{bmatrix}\right\|_{\MM_n(\L(X))}=\left\|\begin{bmatrix}\begin{bmatrix} 0 & 0 \\
r_{x_{ij}} & 0\end{bmatrix}\end{bmatrix}\right\|_{\MM_n(\M(X))},$$
so that the canonical embedding
 $$\iota_X : X \hookrightarrow \M(X), \qquad \iota_X: x \mapsto
 \begin{bmatrix} 0 & 0 \\ r_x & 0
 \end{bmatrix}$$
becomes a complete isometry. This structure is called the \emph{canonical operator space structure} on $X$ (for details we refer to \cite[Section~8.2]{BLM}). Further, since
the canonical embeddings
$$\iota_{M(A)}: M(A) \hookrightarrow \M(X), \qquad \iota_{M(A)} : v \mapsto \begin{bmatrix}
T_v & 0 \\
0 & 0
\end{bmatrix}$$
and
$$\iota_{\BB(X)}: \BB(X) \hookrightarrow \M(X), \qquad \iota_{\BB(X)}: u \mapsto \begin{bmatrix}
0 & 0 \\
0 & u
\end{bmatrix}$$
are injective $*$-homomorphisms between $C^*$-algebras, they are also completely isometric.

We record the next simple fact:
\begin{lemma}\label{lem:ind}
Let $X$ be a Hilbert $A$-module. For each $\phi \in \CB(X)$ we define a map
$$\wt{\phi}: \M(X) \to \M(X) \qquad \mbox{by}  \qquad
\wt{\phi}\left(\begin{bmatrix}
T_v & l_y \\
r_x & u \\
\end{bmatrix}\right):=\begin{bmatrix}
0 & 0 \\
r_{\phi(x)} & 0 \\
\end{bmatrix}.
$$
Then $\wt{\phi}\in \CB(\M(X))$ and $\|\wt{\phi}\|_{cb}=\|\phi\|_{cb}$.
\end{lemma}

\begin{proof}
For all $n \in \N$, $\begin{bmatrix} v_{ij} \end{bmatrix}\in \MM_n(M(A))$, $\begin{bmatrix} x_{ij}  \end{bmatrix}, \begin{bmatrix} y_{ij}  \end{bmatrix}\in \MM_n(X)$ and $\begin{bmatrix} u_{ij} \end{bmatrix} \in \MM_n(\BB(X))$ we have
\begin{eqnarray*}
\left\|\wt{\phi}_n \left(\begin{bmatrix}\begin{bmatrix}T_{v_{ij}} & l_{y_{ij}} \\
r_{x_{ij}} & u_{ij}\end{bmatrix}\end{bmatrix}\right)\right\|_{\MM_n(\M(X))} &=&
\left\| (\iota_X)_n \left(\begin{bmatrix} \phi(x_{ij})\end{bmatrix}\right)\right\|
_{\MM_n(\M(X))} \\
&=&\left\|\begin{bmatrix} \phi(x_{ij})\end{bmatrix}\right\|_{\MM_n(X)}\\
&=&\left\|\phi_n\left(\begin{bmatrix} x_{ij}\end{bmatrix}\right)\right\|_{\MM_n(X)}.
\end{eqnarray*}
\end{proof}

\begin{remark}\label{rem:hbim}
By Remark~\ref{rem:multmod} any Hilbert $A$-module $X$ can be considered as a Hilbert $M(A)$-module and every $u \in \BB(X)$ is $M(A)$-linear. Now for all $u \in \BB(X)$, $x \in X$ and $v \in M(A)$ we have $u(xv)=(ux)v$, so in this way $X$ becomes a Banach $\BB(X)-M(A)$-bimodule (in particular, the product $uxv$ is unambiguously defined). Moreover, it is straightforward to check that each matrix space $\MM_n(X)$ ($n\in \N$) is a Banach $\MM_n(\BB(X))-\MM_n(M(A))$-bimodule in the canonical way. That is,
$$
\left\|\begin{bmatrix}u_{ij}\end{bmatrix} \begin{bmatrix}x_{ij}\end{bmatrix}\right\|_{\MM_n(X)} \leq \left\|\begin{bmatrix}u_{ij}\end{bmatrix}\right\|_{\MM_n(\BB(X))} \left\|\begin{bmatrix}x_{ij}\end{bmatrix}\right\|_{\MM_n(X)}
$$
and
$$
\left\|\begin{bmatrix}x_{ij}\end{bmatrix}\begin{bmatrix}v_{ij}\end{bmatrix}\right\|_{\MM_n(X)} \leq \left\|\begin{bmatrix}x_{ij}\end{bmatrix}\right\|_{\MM_n(X)}\left\|\begin{bmatrix}v_{ij}\end{bmatrix}\right\|_{\MM_n(M(A))}
$$
for all $n \in \N$, $\begin{bmatrix}u_{ij}\end{bmatrix} \in \MM_n(\BB(X))$, $\begin{bmatrix}v_{ij}\end{bmatrix} \in \MM_n(M(A))$ and $\begin{bmatrix}x_{ij}\end{bmatrix} \in \MM_n(X)$.
\end{remark}
\smallskip

Let us now introduce the class of elementary operators on Hilbert $C^*$-modules.

If $X$ is a Hilbert $A$-module, then first, following the $C^*$-algebraic case, for each $u \in \BB(X)$ and $v \in M(A)$ we define a map
$$M_{u,v} : X \to X \qquad \mbox{by} \qquad M_{u,v} : x \mapsto uxv.$$

\begin{definition}
By an \emph{elementary operator} on a Hilbert $A$-module $X$ we mean a map $\phi : X \to X$ for which there exists a finite number of elements $u_1, \ldots, u_k \in \BB(X)$ and $v_1, \ldots, v_k \in M(A)$ such that
\begin{equation}\label{eq:elop}
\phi=\sum_{i=1}^k M_{u_i,v_i}.
\end{equation}
\end{definition}

\begin{example}
If a $C^*$-algebra $A$ is considered as a Hilbert $A$-module in the standard way, then $\BB(A)$ and $M(A)$ coincide, so elementary operators on $A$, as a Hilbert $A$-module, agree with the usual notion of elementary operators on $A$.
\end{example}

Similarly as in the $C^*$-algebraic case, if $X$ is a Hilbert $A$-module, then using the operator space axioms, Remark \ref{rem:hbim} and the $C^*$-identity, it is easy to verify that elementary operators on $X$ are completely bounded and that their cb-norm is dominated by the Haagerup norm of their corresponding tensor in $\BB(X)\otimes M(A)$. That is, if an elementary operator $\phi:X \to X$ is represented as in (\ref{eq:elop}) then
$$
\left\|\phi\right\|_{cb}\leq \left\|\sum_{i=1}^k u_i \otimes v_i\right\|_h
$$
(see \cite[p.~207]{AM}). Therefore, the mapping
\[
(\BB(X) \otimes M(A), \|\cdot\|_h) \to (\CB(X), \|\cdot\|_{cb}) \quad \mbox{given by} \quad \sum_{i=1}^k u_i \otimes v_i \mapsto \sum_{i=1}^k M_{u_i,v_i},
\]
is a well-defined contraction, so we can continuously extend it to the map
$$\Theta_X:(\BB(X) \otimes_h M(A), \|\cdot\|_h) \to (\CB(X), \|\cdot\|_{cb}),$$
where $\BB(X)\otimes_h M(A)$ is the completion of $\BB(X)\otimes M(A)$ with respect to $\|\cdot\|_h$.

\begin{lemma}\label{lem:emb}
Using the same notation as in Lemma \ref{lem:ind}, for each $t \in \BB(X)\otimes_h M(A)$ we have
$$\wt{\Theta_X(t)}=\Theta_{\M(X)}((\iota_{\BB(X)} \otimes \iota_{M(A)})(t)).$$
\end{lemma}
\begin{proof}
By \cite[Proposition~1.5.6]{BLM} there exist sequences $(u_k)$ in $\BB(X)$ and $(v_k)$ in $M(A)$ such that the series $\sum_{k=1}^\infty u_k u_k^*$ and $\sum_{k=1}^\infty v_k^*v_k$ are norm convergent and $t=\sum_{k=1}^\infty u_k \otimes v_k$. Then the series $\sum_{k=1}^\infty u_k xv_k$ is norm convergent for every $x\in X$ and for all $v \in M(A)$, $x,y \in X$ and $u \in \BB(X)$ we have
\begin{eqnarray*}
\wt{\Theta_X(t)}\left(\begin{bmatrix}
T_v & l_y \\
r_x & u
\end{bmatrix}\right)&=&\begin{bmatrix}
0 & 0 \\
\sum_{k=1}^\infty r_{u_k x v_k} & 0
\end{bmatrix} = \sum_{k=1}^\infty \begin{bmatrix} 0 & 0 \\
0 & u_k \end{bmatrix}
\begin{bmatrix}
T_v & l_y \\
r_x & u \\
\end{bmatrix}\begin{bmatrix}
T_{v_k} & 0 \\
0 & 0
\end{bmatrix}\\
&=&\Theta_{\M(X)}((\iota_{\BB(X)} \otimes \iota_{M(A)})(t))\left(\begin{bmatrix}
T_v & l_y \\
r_x & u
\end{bmatrix}\right).
\end{eqnarray*}
\end{proof}

We are now ready to prove the main result of this paper, the generalization of Theorem \ref{thetaiso} in the context of Hilbert $C^*$-modules.

\begin{theorem}\label{thm:main}
Let $X$ be a non-zero Hilbert $A$-module. The following conditions are equivalent:
\begin{itemize}
\item[(i)] $\Theta_X$ is isometric.
\item[(ii)]  $\Theta_X$ is injective.
\item[(iii)] $A$ is a prime $C^*$-algebra.
\end{itemize}
\end{theorem}
\begin{proof}
(i) $\Longrightarrow$ (ii). This is trivial.

\smallskip

(ii)$\Longrightarrow$ (iii). Assume that $A$ is not prime. Then by Proposition  \ref{prop:prime} there are non-zero elements $u \in \KK(X)$ and $a \in A$ such that $uxa=0$ for all $x \in X$. Then $u \otimes a$ is a non-zero tensor in $\KK(X)\otimes A \subseteq \BB(X)\otimes M(A)$ but
$$\Theta_X(u \otimes a)(x)=uxa=0$$
for all $x \in X$.

\smallskip

(iii) $\Longrightarrow$ (i). 
Since the canonical embeddings $\iota_{\BB(X)}: \BB(X) \hookrightarrow \M(X)$ and $\iota_{M(A)}: M(A) \hookrightarrow \M(X)$ are completely isometric, the injectivity of the Haagerup tensor product implies
$$\|(\iota_{\BB(X)} \otimes \iota_{M(A)})(t)\|_h=\|t\|_h \qquad \forall t \in \BB(X)\otimes_h M(A)$$
(see e.g. \cite[Section~1.5.5]{BLM}). If $A$ is a prime $C^*$-algebra, then by Proposition \ref{prop:prime} $\M(X)$ is also prime, so Theorem \ref{thetaiso} implies
$$\|\Theta_{\M(X)}(t')\|_{cb}=\|t'\|_h \qquad \forall t' \in \M(X) \otimes_h \M(X).$$
Then using Lemmas \ref{lem:ind} and \ref{lem:emb} we see that for all $t \in \BB(X) \otimes_h M(A)$ we have
\begin{eqnarray*}
\|\Theta_{X}(t)\|_{cb}&=&\|\widetilde{\Theta_{X}(t)}\|_{cb}=\|\Theta_{\M(X)}((\iota_{\BB(X)} \otimes \iota_{M(A)})(t))\|_{cb}\\
&=&\|(\iota_{\BB(X)} \otimes \iota_{M(A)})(t)\|_{h} =\|t\|_h.
\end{eqnarray*}
Thus, $\Theta_X$ is isometric.
\end{proof}


\section*{Acknowledgements}
We thank Professor Michael Frank for reading the manuscript and for his comments. We also thank the anonymous referee for useful suggestions that helped us to improve the presentation of our results.

\end{document}